\documentclass{mpi2015-cscpreprint}

\usepackage[american]{babel}
\usepackage{graphicx,subfig,float,pgfplots}
\usepackage{amssymb,amsthm,amsmath}
\usepackage{multicol}
\usepackage{makecell,tabularx}

\newtheorem{theorem}{Theorem}[section]

\newtheorem{lemma}[theorem]{Lemma}

\newtheorem{definition}{Definition}[section]
\newtheorem{remark}{Remark}[section]

\newtheorem*{theorem*}{Theorem}

\numberwithin{equation}{section}

\begin{document}

\title{Stability and decay rate estimates for a nonlinear dispersed flow reactor model with boundary control}
  
\author[$1$]{Yevgeniia~Yevgenieva}
\affil[$1$]{Max Planck Institute for Dynamics of Complex Technical Systems,
              39106 Magdeburg, Germany\\ \newline
              Institute of Applied Mathematics and Mechanics, National Academy of Sciences of Ukraine, \newline 84116 Sloviansk, Ukraine\authorcr
  \email{yevgenieva@mpi-magdeburg.mpg.de}, \orcid{0000-0002-1867-3698}}
  
\author[$2$]{Alexander~Zuyev}
\affil[$2$]{Max Planck Institute for Dynamics of Complex Technical Systems,
              39106 Magdeburg, Germany\\ \newline
              Institute of Applied Mathematics and Mechanics, National Academy of Sciences of Ukraine, \newline 84116 Sloviansk, Ukraine\authorcr
  \email{zuyev@mpi-magdeburg.mpg.de}, \orcid{0000-0002-7610-5621}}

\author[$3$]{\, Peter~Benner}
\affil[$3$]{Max Planck Institute for Dynamics of Complex Technical Systems,
              39106 Magdeburg, Germany\authorcr
  \email{benner@mpi-magdeburg.mpg.de}, \orcid{0000-0003-3362-4103}}
  
\shorttitle{Stability and decay rate...}
\shortauthor{Ye.~Yevgenieva, A.~Zuyev, P.~Benner}
\shortdate{}
  

  
\abstract{We investigate a nonlinear parabolic partial differential equation whose boundary conditions contain a single control input. This model describes a chemical reaction of the type ``$A \to $ product'', occurring in a dispersed flow tubular reactor. The existence and uniqueness of solutions to the nonlinear Cauchy problem under consideration are established by applying the theory of strongly continuous semigroups of operators. We also prove the stability of the equilibrium of the closed-loop system with a proposed feedback law. Additionally, using Lyapunov's direct method, we evaluate the exponential decay rate of the solutions.}

\novelty{The key contribution of our work is the following:
\begin{itemize}
\item A class of feedback controllers is designed to achieve exponential stability of the steady-state solution in the closed-loop system.
\item The existence and uniqueness of solutions of the model with the proposed controllers are proved.
\item An evaluation of the solution's decay rate is presented.
\item Numerical simulations are presented to illustrate the large-time behavior of  solutions to the closed-loop system under consideration.
\end{itemize}}

\maketitle

\section{Introduction}

In recent years, there has been significant research into chemical reaction models, particularly in the context of boundary control, which provides a natural means of regulating the process. In tubular reactors, boundary control focuses on manipulating inlet concentrations, temperatures, and flow rates to adjust the reaction rate and maintain the desired process performance.

The development of control strategies for optimizing and stabilizing chemical reactions is crucial in chemical engineering. Various papers \cite{BosKr2002,SH13,YevZuBenS-M2024,ZS-MB2017} have explored optimal control strategies aimed at maximizing product yield while utilizing a fixed amount of reactants. 

Alongside optimization, the qualitative properties of the system, such as the existence and uniqueness of solutions, as well as stability, are essential for ensuring consistent and reliable operation under varying conditions. Stability is particularly significant as it ensures the system's ability to rapidly converge to a stable operating state following perturbations or disturbances, contributing to improved process efficiency, product quality, and safety.

In this paper, we consider the mathematical model of a chemical reaction conducted in a dispersed flow tubular reactor (DFTR). Earlier studies have explored the nonlinear non-isothermal DFTR model, which accounts for temperature variations. In particular, the existence and uniqueness of solutions for this model were established in \cite{LaabAchWinDoch2001}. Furthermore, exponential stability for this model was proven in \cite{HasWinDoch2020} using nonlinear semigroup theory and the analysis of $m$-dissipative operators. 

Moreover, the exponential stabilization of several classes of linear and nonlinear parabolic equations with boundary control has been investigated in numerous papers (see, e.g., \cite{WangWuSun2014,WangWuSun2017,WeiLi2021} and references therein).

Our study considers a nonlinear isothermal DFTR model with a chemical reaction of arbitrary nonnegative order $n$, in contrast to the case $n=1$ studied in \cite{HasWinDoch2020, LaabAchWinDoch2001}.

\section{Basic notations and auxiliary results}

This section presents the notations and auxiliary theorems used throughout the paper.

\vskip 2mm
\noindent\textbf{Basic notations}
\vskip 2mm
\hskip -6mm
\begin{tabularx}{\textwidth}{rX}
$C^{2}[a,b]$ & the set of continuous functions on $[a,b]$ with continuous derivatives up to the second order;\\
$L^{1}(a,b)$ & the space of all measurable real functions on $(a,b)$ for which the Lebesgue integral $\int_a^b f(x)\,dx$ exists and is finite;\\
$L^{2}(a,b)$ & the Hilbert space of all measurable real functions on $(a,b)$ whose squared absolute values are Lebesgue integrable;\\
$||f||_{L^{2}(a,b)}$ & $:=\left(\int_a^b|f(x)|^{2}\,dx\right)^{\frac{1}{2}}$, the norm in $L^{2}(a,b)$; \\
$\left<f,g\right>_{L^2(a,b)}$ & $:=\int_{a}^{b}f(x)g(x)\,dx$, the inner product in $L^2(a,b)$;\\
$L^{\infty}(a,b)$ & the space of all measurable real functions on $(a,b)$ that are bounded almost everywhere on $(a,b)$;\\
$H^{2}(a,b)$ & the Hilbert space of functions in $L^{2}(a,b)$ whose weak derivatives up to the second order also belong to $L^{2}(a,b)$.
\end{tabularx}

In this paper, we will employ the theory of strongly continuous semigroups ($C_0$ semigroups). For an introduction, we refer the reader to~\cite{Pazy}. The following theorems serve as the main tools in our analysis.

\begin{theorem}[\textbf{Lumer--Phillips theorem \cite[Section 1.4]{Pazy}}]\label{th0}
Let $\mathcal{A}: D(A)\to X$ be a linear operator defined on a subset $D(\mathcal{A})$ of a Banach space $X$. Then $\mathcal{A}$ generates a strongly continuous semigroup of operators on $X$ if and only if:
\begin{itemize}
    \item[(i)] $D(\mathcal{A})$ is dense in $X$;
    \item[(ii)] $\mathcal{A}$ is dissipative;
    \item[(iii)] $R(\mathcal{A}-\lambda\mathcal{I})=X$ for some $\lambda>0$, where $R$ denotes the range of the operator and $\mathcal{I}$ denotes the identity operator.
\end{itemize}
\end{theorem}

\begin{theorem}[\textbf{\cite[Chapter 2, Proposition 5.3]{LiYong}}]\label{LiYong}
Let the linear operator $\mathcal{A}:D(\mathcal{A})\to X$ generate a $C_0$ semigroup $\{e^{t\mathcal{A}}\}_{t\geqslant0}$ and let $f:X\to X$ be the operator satisfying the following:
\begin{equation}\label{li_yong_cond}
\begin{aligned}
\|f(x_1)-f(x_2)\|_X&\leqslant \mu(t)\|x_1-x_2\|_X,\quad&&\forall\,t\in[0,T],\ \forall\,x_1,x_2\in X,
\\\|f(0)\|_X&\leqslant \mu(t),\quad&&\forall\,t\in[0,T],
\end{aligned}
\end{equation}
with some function $\mu\in L^1(0,T)$. Then, the equation
\begin{equation*}
y(t)=e^{t\mathcal{A}}y_0+\int_0^t e^{(t-s)\mathcal{A}}f(y(s))ds,\quad t\in[0,T]
\end{equation*}
admits a unique solution $y$ for any $y_0\in X$.
\end{theorem}

\section{Description of the model}

Consider the mathematical model of a DFTR represented by the nonlinear parabolic equation with the following Danckwerts boundary conditions~\cite{NauMal1983,SH13}:
\begin{equation}\label{dftr}
\begin{aligned}
    \frac{\partial C_A}{\partial t}&=D_{ax}\frac{\partial^2C_A}{\partial x^2} -v\frac{\partial C_A}{\partial x}-kC_A^n, \qquad (x,t)\in(0,l)\times(0,T),
\\ C_A(0,t)&=C_{A_0}(t)+\frac{D_{ax}}{v}\frac{\partial C_A}{\partial x}\Big|_{x=0}, \quad
\frac{\partial C_A}{\partial x}\Big|_{x=l}=0,
\end{aligned}
\end{equation}
where $C_A(x,t)\geqslant0$ is the reactant $A$ concentration inside the reactor at the distance $x$ from the inlet and at time $t$, $l$ is the length of the reactor tube, $C_{A_0}(t)$ is the concentration of $A$ in the inlet stream (that also contains another inert component), $n>0$ is the reaction order, $D_{ax}>0$ is the axial dispersion coefficient, $v>0$ is the flow-rate of the reaction stream, and $k>0$ is the reaction rate constant. The function $C_{A_0}(t)\geqslant0$ is treated as the control input.

In order to rewrite the problem in an abstract form, we present the steady-state solution $\overline{C}_A\in C^2[0,l]$ as a solution to the problem:
\begin{equation}\label{dftr-ss}
\begin{aligned}
    D_{ax}\frac{d^2\overline{C}_A}{dx^2}&=v\frac{d \overline{C}_A}{d x}+k\overline{C}_A^n,\qquad x \in(0,l),
\\ \overline{C}_A(0)&=\overline{u} +\frac{D_{ax}}{v}\frac{d \overline{C}_A}{dx}\Big|_{x=0}, \qquad
\frac{d \overline{C}_A}{dx}\Big|_{x=l}=0,
\end{aligned}
\end{equation}
where $\overline{u} = \text{const} > 0$ denotes the steady-state control, selected to ensure that the solution $\overline{C}_A(x)$ remains non-negative, i.e., $\overline{C}_A(x) \geqslant 0$ for all $x \in [0, l]$.

In the following discussion, we reformulate the problem expressed in~\eqref{dftr} in terms of deviations by introducing the function $w(x,t)=C_A(x,t)-\overline{C}_A(x)$:
\begin{equation}\label{dftr1}
\begin{aligned}
    \frac{\partial w}{\partial t}&=D_{ax}\frac{\partial^2w}{\partial x^2} -v\frac{\partial w}{\partial x}+k\overline{C}_A^n -k (w+\overline{C}_A)^n,
\qquad (x,t)\in(0,l)\times(0,T),
\\ w(0,t)&=u_w(t)+\frac{D_{ax}}{v}\frac{\partial w}{\partial x}\Big|_{x=0}, \quad
\frac{\partial w}{\partial x}\Big|_{x=l}=0,
\end{aligned}
\end{equation}
where $u_w(t):=C_{A_0}(t)-\overline{u}$. 

We will investigate the above problem with a boundary feedback control $u_w$ in the form
\begin{equation}\label{control}
u_w(t)=\alpha\,w(0,t),\quad \alpha\in\left[0,\tfrac{1}{2}\right].
\end{equation}

\begin{remark}
The parameterized class of controls \eqref{control} also includes the particular case $u_w(t)\equiv0$, which represents the steady-state control $\overline{u}$ for the initial problem \eqref{dftr}.
\end{remark}

Since the concentration $C_A(x,t)$ is nonnegative and bounded above by a constant due to physical considerations, the reaction model -- particularly the reaction rate -- should also satisfy certain constraints in terms of the deviation $w(x,t)$. For this purpose, we introduce the saturation function of $w$ as follows:
\begin{equation*}
\text{Sat}_M(w):=
\begin{cases}
w,\quad &\text{if }|w|\leqslant M,
\\ M,\quad &\text{if }w>M,
\\-M, \quad &\text{if }w<-M,
\end{cases}
\end{equation*}
where $M<\infty$ is a positive constant.

To study the well-posedness of the considered model, we rewrite system~\eqref{dftr1} in an abstract form.
First, we introduce the operator $\mathcal{A}:D(\mathcal{A})\to X$ such that
\begin{equation}\label{opA}
    \mathcal{A}:{\xi}\in D(\mathcal{A}) \mapsto \mathcal{A}{\xi} =
            (D_{ax}\xi'' -v\xi')\in X,
\end{equation}
where $X=L^2(0,l)$, $\xi'(x):=\frac{d\xi}{dx}$ denotes the derivative, and, for any $\alpha\in\left[0,\tfrac{1}{2}\right]$,
\begin{equation*}
D(\mathcal{A})=D_\alpha(\mathcal{A}):=\left\{\xi\in H^2(0,l):(1-\alpha)\xi(0)=\frac{D_{ax}}{v}\xi'(0),\ 
\xi'(l)=0\right\},
\end{equation*}
where $\xi'(x_0):=\frac{d\xi}{dx}\big|_{x=x_0}$.

Consider also the nonlinear operator $r:X\to X$ defined by the following relation:
\begin{equation*}
    r({\xi})=k\overline{C}_A^n-k (\text{Sat}_M(\xi)+\overline{C}_A)^n.
\end{equation*}
Now, the problem \eqref{dftr1} can be rewritten in the abstract form for $\xi(t)=w(\cdot,t)$ as follows:
\begin{equation}\label{eq1.5}
\begin{aligned}
 \dot{{\xi}}(t)&=\mathcal{A}\,{\xi}(t)+r({\xi}(t)), \\\xi(0)&=\xi_0\in D(\mathcal{A}).   
\end{aligned}
\end{equation}

\section{Main results}

\subsection{Existence and uniqueness of solution to the problem \eqref{eq1.5}}

We use the strongly continuous semigroup theory to prove the existence and uniqueness of the mild solution to the problem \eqref{eq1.5}. First, we obtain the semigroup generation property for the linear operator $\mathcal{A}$.

\begin{lemma}\label{th1}
The operator $\mathcal{A}:D(\mathcal{A})\to X$, defined by the relation \eqref{opA}, generates the $C_0$--semigroup $\{e^{t\mathcal{A}}\}_{t\geqslant0}$ of bounded linear operators on $X$.
\end{lemma}

\begin{proof}
The proof of this lemma boils down to verifying the conditions of Theorem~\ref{th0}.

First, we check the condition $(ii)$ and show that the operator $\mathcal{A}$ is dissipative, namely, that $\left<\mathcal{A}\xi,\xi\right>_{L^2(0,l)}\leqslant0$ $\forall\,\xi\in D(\mathcal{A})$: 
\begin{equation*}
\begin{aligned}
\left<\mathcal{A}\xi,\xi\right>_{L^2(0,l)}&=\int_0^l (D_{ax}\xi'' -v\xi')\xi\,dx
=D_{ax}\xi\xi'\big|_0^l-D_{ax}\int_0^l(\xi')^2dx-\frac{v}{2}\int_0^l(\xi^2)'dx
\\&=-D_{ax}\xi(0)\xi'(0)-D_{ax}\int_0^l(\xi')^2dx-\frac{v}{2}\xi^2(l)+\frac{v}{2}\xi^2(0)
\\&=-v\left(\frac{1}{2}-\alpha\right)\xi^2(0)-D_{ax}\int_0^l(\xi')^2dx-\frac{v}{2}\xi^2(l)\leqslant0
\end{aligned}
\end{equation*}
for all $\xi\in D(\mathcal{A})$.

Now, we prove that $D(\mathcal{A})$ is dense in $L^2(0,l)$ (condition $(i)$). 
First, we introduce the set 
$$D_0:=\left\{\xi\in H^2(0,l):\xi(0)=\xi'(0)=\xi'(l)=0\right\}$$
and prove that $D_0$ is dense in $L^2(0,l)$. Then, the obvious embedding $D_0\subset D(\mathcal{A})\subset L^2(0,l)$ ensures that $D(\mathcal{A})$ is dense in $L^2(0,l)$.

The denseness proof is similar to the proof of Corollary~4.23 in~\cite{brezis2010}.
Given $\xi\in L^2(0,l)$, we set
\begin{equation*}
\bar{\xi}=
\begin{cases}
\xi(x),\quad&\text{if }x\in(0,l),
\\0,&\text{if }x\in\mathbb{R}\setminus(0,l),
\end{cases}
\end{equation*}
so that $\bar\xi\in L^2(\mathbb{R})$.

Let $\{K_n\}$ be a sequence of intervals such that $K_n:=\left(\frac{2}{n},l-\frac{2}{n}\right)$. 
Let $\{\rho_n\}$ be a sequence of mollifiers such that
\begin{equation*}
\rho_n\in C^{\infty}_c(\mathbb{R}),\quad \text{supp}\,\rho_n\subset\left(-\frac{1}{n},\frac{1}{n}\right),
\quad\int_{\mathbb{R}}\rho_n(x)dx=1, \quad \rho_n(x)\geqslant0\ \forall\,x\in\mathbb{R}.    
\end{equation*}
Set $g_n=\chi_{K_n}\bar\xi$ and
\begin{equation*}
\xi_n=\rho_n * g_n:=\int_{\mathbb{R}}\rho_n(x-y)\,g_n(y)dy,
\end{equation*}
so that $\text{supp}\,\xi_n\subset\left(\frac{1}{n},l-\frac{1}{n}\right)$. It follows that $\xi_n\in C^{\infty}_c(\mathbb{R})$ and, moreover, $\xi_n(0)=\xi'_n(0)=\xi'_n(l)=0$, so $\xi_n\in D_0$.

On the other hand, we have
\begin{equation*}
||\xi_n-\xi||_{L^2(0,l)}=||\xi_n-\bar\xi||_{L^2(\mathbb{R})}
\leqslant ||(\rho_n * g_n)-(\rho_n * \bar\xi)||_{L^2(\mathbb{R})}+||(\rho_n * \bar\xi)-\bar\xi||_{L^2(\mathbb{R})}.
\end{equation*}
By Young's convolution equation, we get
\begin{equation*}
||(\rho_n * g_n)-(\rho_n * \bar\xi)||_{L^2(\mathbb{R})}\leqslant ||\rho_n||_{L^1(\mathbb{R})} ||g_n-\bar\xi||_{L^2(\mathbb{R})},
\end{equation*}
which leads to 
\begin{equation*}
||\xi_n-\xi||_{L^2(0,l)}\leqslant ||g_n-\bar\xi||_{L^2(\mathbb{R})}+||(\rho_n * \bar\xi)-\bar\xi||_{L^2(\mathbb{R})}.
\end{equation*}
We note that $||g_n-\bar\xi||_{L^2(\mathbb{R})}\to0$ according to the definition of the sequence $g_n$ and $||(\rho_n * \bar\xi)-\bar\xi||_{L^2(\mathbb{R})}\to0$ by Theorem~4.22 from~\cite{brezis2010}.
Then we conclude that  $||\xi_n-\xi||_{L^2(0,l)}\to 0$ as $n\to\infty$.

Finally, we check the condition $(iii)$. Specifically, we need to prove that 
\begin{equation*}
\forall\,\eta\in X\ \exists\,\lambda>0,\ \xi\in D(\mathcal{A}):\ \mathcal{A}\xi-\lambda\xi=\eta.
\end{equation*}
We start with solving equation $\mathcal{A}\xi-\lambda\xi=\eta$ for an arbitrary $\eta\in X$:
\begin{equation}\label{eq1.6}
D_{ax}\xi''(x) -v\xi'(x)-\lambda\xi(x)=\eta(x).
\end{equation}
The solution takes the following form:
\begin{equation*}
\xi(x)=C_1(x)e^{\nu_1x}+C_2(x)e^{\nu_2x},
\end{equation*}
where $\nu_1$ and $\nu_2$ are the roots of the characteristic equation $D_{ax}\nu^2 -v\nu-\lambda=0$ with some $\lambda>0$, $\nu_1\neq\nu_2$.
Here, $C_1(\cdot)$, $C_2(\cdot)$ are the functions satisfying the system:
\begin{equation*}
\begin{cases}
C_1'(x)e^{\nu_1x}+C_2'(x)e^{\nu_2x}=0\\
\nu_1C_1'(x)e^{\nu_1x}+\nu_2C_2'(x)e^{\nu_2x}=\frac{\eta(x)}{D_{ax}}.
\end{cases}
\end{equation*}
After solving the latter system, we obtain the solution in the form
\begin{equation}\label{eq1.7}
\xi(x)=\frac{1}{D_{ax}(\nu_2-\nu_1)}\int_0^x\eta(s)\left[e^{\nu_2(x-s)}-e^{\nu_1(x-s)}\right]ds
+C_3e^{\nu_1x}+C_4e^{\nu_2x},
\end{equation}
where the constants $C_3$ and $C_4$ are defined by the  boundary conditions and satisfy the algebraic system:
\begin{equation*}
\begin{cases}
\left(\tfrac{D_{ax}}{v}\nu_1-1+\alpha\right)C_3+\left(\tfrac{D_{ax}}{v}\nu_2-1+\alpha\right)C_4=0,
\\\nu_1e^{\nu_1l}C_3+\nu_2e^{\nu_2l}C_4=\frac{I(l)}{D_{ax}(\nu_1-\nu_2)},
\end{cases}
\end{equation*}
where $I(x):=\int_0^x\eta(s)\left[\nu_2e^{\nu_2(x-s)}-\nu_1e^{\nu_1(x-s)}\right]ds$. Note that
\begin{equation*}
\xi'(x)=\frac{1}{D_{ax}(\nu_2-\nu_1)}\int_0^x\eta(s)\left[\nu_2e^{\nu_2(x-s)}-\nu_1e^{\nu_1(x-s)}\right]ds
+C_3\nu_1e^{\nu_1x}+C_4\nu_2e^{\nu_2x},
\end{equation*}
so $\xi''\in L^2(0,l)$, which means that $\xi\in H^2(0,l)$ for all functions $\eta\in L^2[0,l]$. Moreover, the solution $\xi$ is constructed to satisfy the imposed boundary conditions, so it belongs to $D(\mathcal{A})$. Now we can state that, for every function $\eta\in X$, the function $\xi$ defined by the formula \eqref{eq1.7} belongs to $D(\mathcal{A})$.
\end{proof}

Now, we introduce the definition of a mild solution to problem~\eqref{eq1.5}.

\begin{definition}\label{def_mild}
A function $\xi\in C([0,T];X)$ is called a \textbf{mild solution} of \eqref{eq1.5} if it satisfies the following equation:
\begin{equation}\label{mild}
\xi(t)=e^{t\mathcal{A}}\xi_0+\int_0^t e^{(t-s)\mathcal{A}}r(\xi(s))ds,\quad t\in[0,T].
\end{equation}
\end{definition}

The main result of this subsection reads as follows.

\begin{theorem}\label{th2}
For each initial function $\xi_0\in X$ and any given $T>0$, there exists a unique mild solution of problem \eqref{eq1.5} in the sense of Definition~\ref{def_mild}.
\end{theorem}

\begin{proof}
The proof of this theorem consists of checking the conditions of Theorem~\ref{LiYong}. To check condition \eqref{li_yong_cond}, we use the mean value theorem for nonlinear operators (see e.g. \cite[Chapter XVII]{KantAkil}. Namely, for any differentiable nonlinear operator $P:Y\to Y$ and any $y_1,y_2\in Y$, the following inequality holds:
\begin{equation*}
\|P(y_1)-P(y_2)\|\leqslant\sup_{0<\theta<1}\|P'(y_1+\theta(y_2-y_1))\|\,\|y_2-y_1\|,
\end{equation*}
where $\|\cdot\|$ denotes the norm in $Y$. First, we find the Gateaux derivative of the operator $r$ at $\xi\in X$ in the direction $\varphi\in K$, where $K:=\{\varphi\in L^\infty(0,l):\|\varphi\|_{L^\infty(0,l)}\leqslant1\}$:
\begin{equation*}
r'(\xi;\varphi)=\lim_{s\to0}\frac{r(\xi+s\varphi)-r(\xi)}{s}
=-kn(\text{Sat}_M(\xi)+\overline{C}_A)^{n-1}\varphi.
\end{equation*}
Estimating the desired supremum of the norm of the derivative, we get:
{\small
\begin{equation*}
\begin{aligned}
\sup_{0<\theta<1}\|r'(\xi_1+\theta(\xi_2-\xi_1))\|_X
&=kn{\sup_{0<\theta<1}\left(\int_0^l\left(\text{Sat}_M\left(\xi_1+\theta(\xi_2-\xi_1\right)+\overline{C}_A\right)^{2(n-1)}dx\right)^\frac{1}{2}}
\\&\leqslant knl^\frac{1}{2}\left(M+\sup_{x\in[0,l]}\overline{C}_A(x)\right)^{n-1}=:\mu_r<\infty.
\end{aligned}
\end{equation*}}
By the mean value theorem and $r(0)=0$, condition \eqref{li_yong_cond} is validated. We complete the proof of the theorem by applying Lemma~\ref{th1} and Theorem~\ref{LiYong}.
\end{proof}

\subsection{Exponential stability of the steady-state solution}

\begin{theorem}\label{th3}
The trivial solution of problem \eqref{dftr1} with the feedback control~\eqref{control} is exponentially stable in the weighted Lebesgue space $L^2_\rho(0,l)$.
\end{theorem}

\begin{proof}
Define the energy functional for system~\eqref{dftr1} as the squared solution norm in the weighted Lebesgue space $L^2_\rho(0,l)$:
\begin{equation}\label{energy}
   \mathcal{E}=\frac{1}{2}\int_0^l \rho(x)w^2(x,t)dx,
\end{equation}
where the weight $\rho\in L^\infty(0,l)$ is a non-negative function to be defined later.

Using integration by parts, we compute the time derivative of the energy along the trajectories of problem~\eqref{dftr1}:
\begin{equation*}\label{dot_energy}
\begin{aligned}
\dot{\mathcal{E}}=&\int_0^l \rho(x)w(x,t)\dot{w}(x,t)dx
=\int_0^l \rho\,w\left[D_{ax}w'' -vw'+k\overline{C}_A^n -k (w+\overline{C}_A)^n\right]dx
\\=&-v(1-\alpha)\rho(0)w^2(0,t)
-\frac{D_{ax}}{2}\int_0^l\rho'(w^2)' dx 
-D_{ax}\int_0^l\rho(w')^2 dx-\frac{v}{2}\int_0^l\rho(w^2)'dx
\\&-k\int_0^l\rho\left(\overline{C}_A-(w+\overline{C}_A)\right)\left(\overline{C}_A^n-(w+\overline{C}_A)^n\right)dx.
\end{aligned}
\end{equation*}
Since we are interested in non-negative solutions to the problem \eqref{dftr}, we can state that $\overline{C}_A\geqslant0$ and $w+\overline{C}_A=C_A\geqslant0$ for almost all $(x,t)\in(0,l)\times(0,T)$. Using the fact that $(a-b)(a^n-b^n)\geqslant0$ for any positive numbers $a,b,n$, we get:
\begin{equation*}\label{dot_energy}
\dot{\mathcal{E}}\leqslant-v(1-\alpha)\rho(0)w^2(0,t)
-\frac{D_{ax}}{2}\int_0^l\rho'(w^2)'dx-\frac{v}{2}\int_0^l\rho(w^2)'dx.
\end{equation*}
We take $\rho$ as a solution to the problem
\begin{equation}\label{weight}
\rho'(x)=-\gamma\rho(x),\quad x\in[0,l],\quad \rho(0)=\rho_0>0
\end{equation}
with some positive constant $\gamma$. Choosing $\gamma=\frac{v}{2D_{ax}}$, we obtain:
\begin{equation*}\label{dot_energy}
\begin{aligned}
\dot{\mathcal{E}}&\leqslant-v(1-\alpha)\rho_0w^2(0,t)-\frac{v}{4}\int_0^l\rho(w^2)'dx
=-v(1-\alpha)\rho_0w^2(0,t)-\frac{v}{4}\rho\,w^2\big|_0^l+\frac{v}{4}\int_0^l\rho'w^2dx
\\&=-v\left(\frac{3}{4}-\alpha\right)\rho_0w^2(0,t)-\frac{v}{4}\rho(l)w^2(l,t)
-\frac{v^2}{8D_{ax}}\int_0^l\rho\,w^2dx
\leqslant -\frac{v^2}{8D_{ax}}\int_0^l\rho\,w^2dx.
\end{aligned}
\end{equation*}
As a result, we get the differential inequality:
\begin{equation}\label{exp_st}
\dot{\mathcal{E}}\leqslant -\frac{v^2}{8D_{ax}}\mathcal{E},
\end{equation}
which proves Theorem~\ref{th3}.
\end{proof}

\begin{remark}
It is easy to show that exponential stability holds for any $\gamma\in\left(0,\frac{v}{D_{ax}}\right)$, as defined in~\eqref{weight}.
Note that weighted $L^2$-norms have been effectively used in the construction of Lyapunov functionals for hyperbolic systems, e.g., in~\cite{bastin2016stability,zuyev2021stabilization}. Here, we have extended this approach for the nonlinear parabolic case. 
\end{remark}

\begin{remark}
Using the differential inequality~\eqref{exp_st}, we derive the following estimate:
\begin{equation}\label{decay-rate}
||w(\cdot,t)||_{L^2_\rho(0,l)} \leqslant 
e^{-\frac{v^2}{16D_{ax}}t}\,||w(\cdot,0)||_{L^2_\rho(0,l)}, \ t\in[0,T].
\end{equation}
Let $\lambda$ denote the exponential decay rate of the solutions. We then have the following lower bound: 
\begin{equation}\label{decay-rate_lambda}
\lambda \geqslant \lambda_T := \frac{v^2}{16D_{ax}},
\end{equation}
where $\lambda_T$ represents the theoretically computed decay rate. It is important to note that the above estimate is global and does not account for the dependence of the individual decay rate for a particular solution on the reaction order $n$
 the rate constant $k$, and the control gain $\alpha$. The sharpness of the obtained bound~\eqref{decay-rate} remains an open theoretical question, with some numerical analysis presented in Section~\ref{sec_num}.
\end{remark}

\section{Numerical simulations}\label{sec_num}

This section contains the results of numerical simulations of problem~\eqref{dftr1} with the control design~\eqref{control} for different values of the reaction order $n$ and feedback gain coefficient $\alpha$.

We choose the following values of physical parameters:
\begin{equation}\label{values}
\begin{aligned}
&k=0.001\ s^{-1}mol^{-1},\quad v=0.01\ m\,s^{-1},\quad T=400\, s,\quad l=1\ m, \quad Pe=4,
\end{aligned}
\end{equation}
where $Pe$ is the P\'eclet number which shows the relation between diffusion and advection time and is defined by the formula 
\begin{equation*}
Pe=\frac{v\,l}{D_{ax}}.
\end{equation*}

To satisfy the boundary conditions, we take the initial function $w(x,0)$ in the form
\begin{equation}\label{initial}
w(x,0)=-\frac{(x-l)^2}{2}+l\frac{lv(1-\alpha)+2D_{ax}}{2v(1-\alpha)}.
\end{equation}

In the case $n=1$, the steady-state solution of problem~\eqref{dftr-ss} is obtained by solving the corresponding linear differential equation:
\begin{equation*}
\overline{C}_A(x)=C_5e^{\frac{v+q}{2D_{ax}}\,x}+C_6e^{\frac{v-q}{2D_{ax}}\,x},
\end{equation*}
where $q=\sqrt{v^2+4D_{ax}k}$, and the constants $C_5$, $C_6$ are chosen to match the boundary conditions:
\begin{equation*}
C_5=-\overline{u}\,\frac{2v(v-q)}{(v+q)^2e^{\frac{q\,l}{D_{ax}}}-(v-q)^2},
\qquad C_6=\overline{u}\,\frac{2v(v+q)}{(v+q)^2-(v-q)^2e^{-\frac{q\,l}{D_{ax}}}}.
\end{equation*}

The 3D plot of the solution of $w(x,t)$ of problem~\eqref{dftr1} with parameters~\eqref{values}, reaction order $n=1$, and $\alpha=0$ is shown in Fig.~\ref{fig_1}. These simulations are performed in MATLAB R2019b with the use of the \texttt{bvp4c} and the \texttt{pdepe} solvers.

\begin{figure}[h!]
\centering\includegraphics[width=0.6\linewidth]{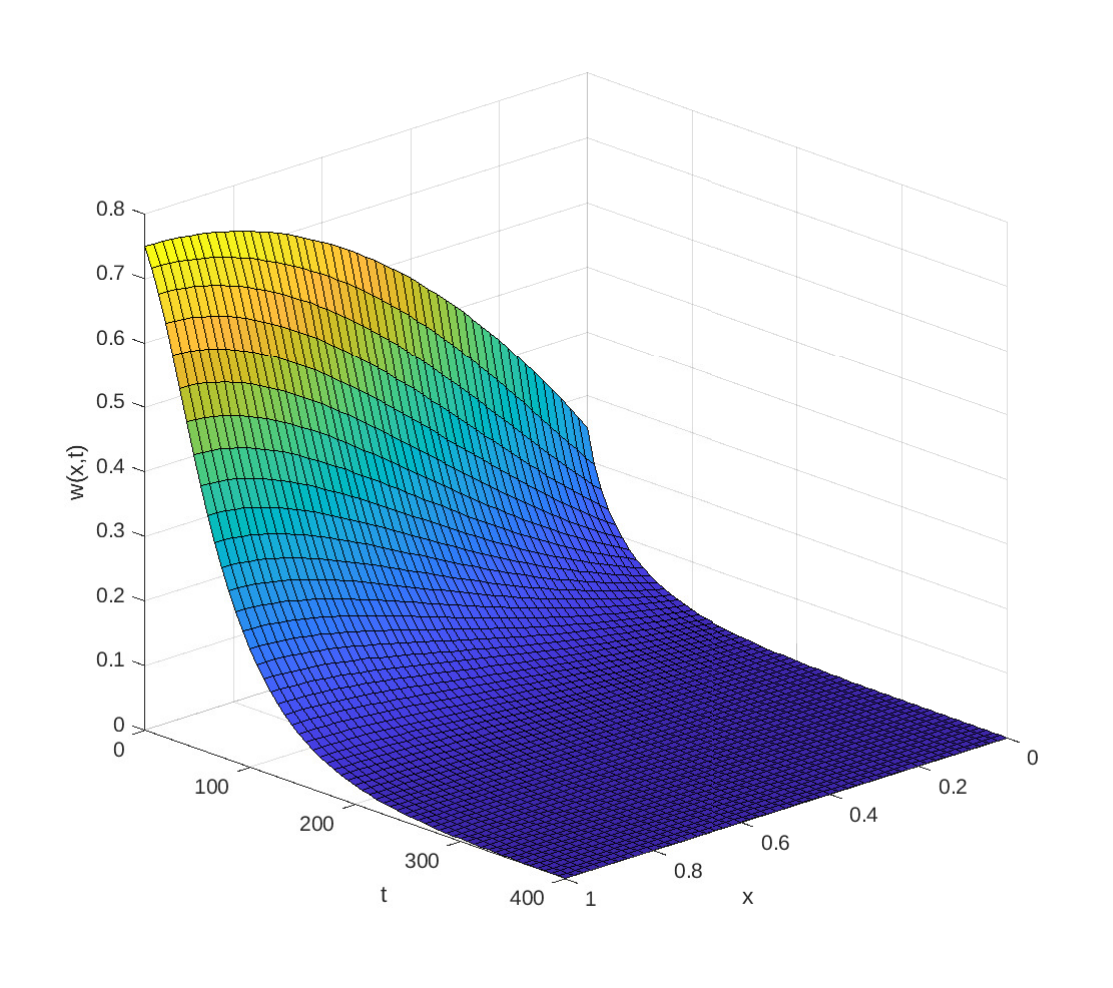}
\vskip-4ex
\caption{3D plot of $w(x,t)$ for $n=1$, $\alpha=0$.} \label{fig_1}
\end{figure}

The solution profiles for $n=\frac12$ and $n=2$, and the corresponding control functions defined by~\eqref{control} are presented in Fig.~\ref{fig_2} and Fig.~\ref{fig_3}, respectively.

\begin{figure}[H]
\begin{minipage}[h]{0.47\linewidth}
\subfloat[Control function $u_w(t)$]{
\includegraphics[width=\linewidth]{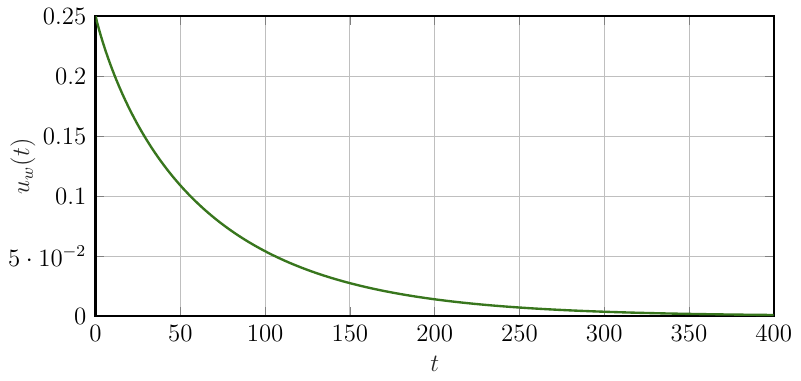}
}
\hspace{0mm}
\subfloat[Graphs of $x$-profiles of $w(x,t)$ at $t=0$ (blue), $t=100$ (orange), $t=200$ (yellow), and $t=300$ (purple)]{
\includegraphics[width=\linewidth]{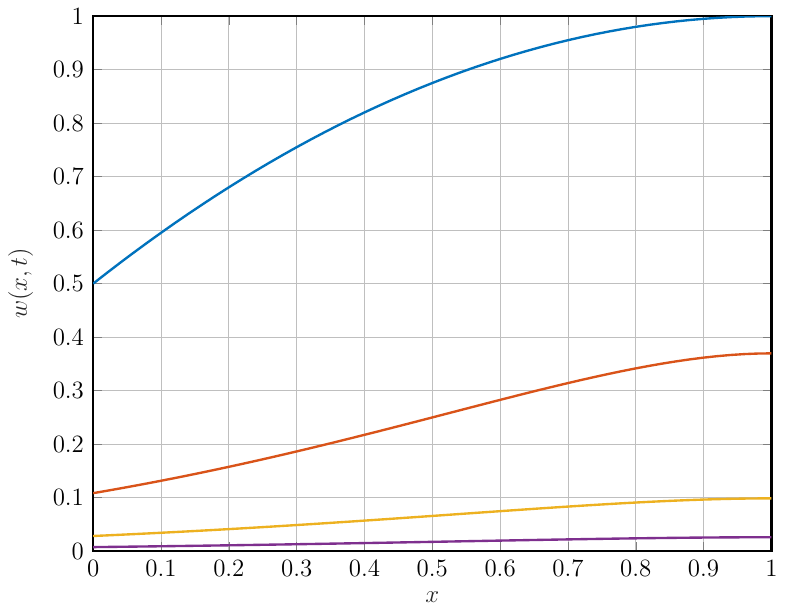}
}
\caption{Case $n=\frac{1}{2}$ and $\alpha=\frac12$.}
\label{fig_2}
\end{minipage}
\hfill
\begin{minipage}[h]{0.47\linewidth}
\subfloat[Control function $u_w(t)$]{
\includegraphics[width=\linewidth]{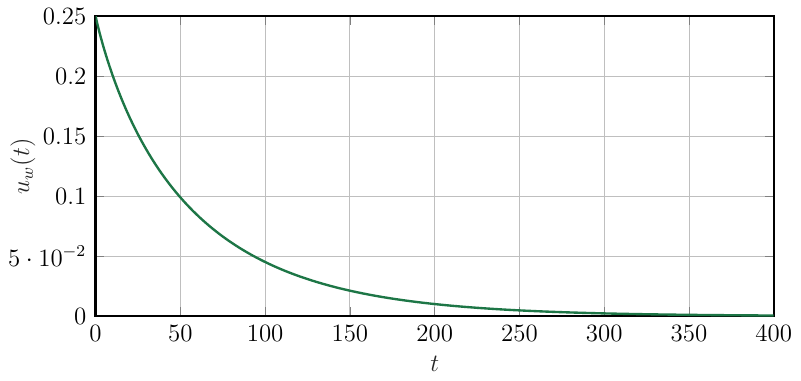}
}
\hspace{0mm}
\subfloat[Graphs of $x$-profiles of $w(x,t)$ at $t=0$ (blue), $t=100$ (orange), $t=200$ (yellow), and $t=300$ (purple)]{
\includegraphics[width=\linewidth]{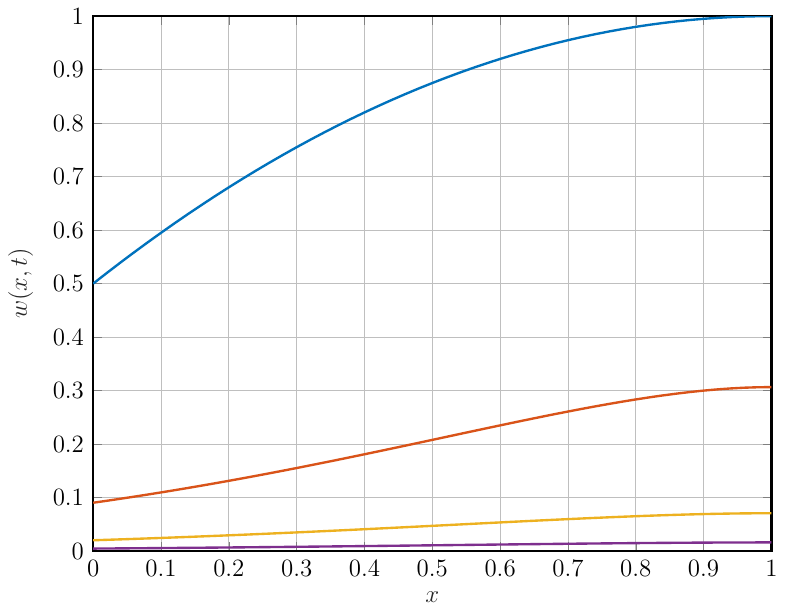}
}
\caption{Case $n=2$ and $\alpha=\frac12$.}
\label{fig_3}
\end{minipage}
\end{figure}

The presented plots illustrate the decay of solutions $w(x,t)$ for large values of $t$. To estimate the behavior of $w(x,t)$ over time, we use standard MATLAB R2019b functions to approximate the decay rate constant as follows:
\begin{equation}\label{lambda_N}
\lambda_N = -\limsup_{t\to\infty} \frac{1}{t}\log \frac{\|w(\cdot,t)\|_{L^2_\rho(0,l)}}{\|w(\cdot,0)\|_{L^2_\rho(0,l)}},
\end{equation}
within the time horizon $t\in[0\,s,7000\,s]$ using a discretization step of $1\,s$.

Table~\ref{table} summarizes the results of numerical computations of the decay rate constant $\lambda_N$ over the range $n\in\{\frac12, 1, 2, 10\}$ and $\alpha\in\{0,\frac14,\frac12\}$. These results enable us to compare the numerical values with the theoretical decay rate, defined in \eqref{decay-rate_lambda} as $\lambda_T = 0.0025$ for the specific parameters given in~\eqref{values}.

\begin{table}[!h]
\caption{Estimated decay-rate $\lambda_N$ for different $\alpha$ and $n$}
\label{table}
\begin{center}
\begin{tabular}{| c |c c c|} 
 \hline
\diaghead(-2,1){aaaaaaa}%
{$\boldsymbol{n}$}{$\boldsymbol{\alpha}$} & $\boldsymbol{0}$ & $\boldsymbol{1/4}$ & $\boldsymbol{1/2}$ \\ [0.5ex] 
 \hline
 $\boldsymbol{1/2}$ & $0.0045$ & $0.0036$ & $0.0039$ \\ 
 \hline
 $\boldsymbol{1}$ & $0.0038$ & $0.0037$ & $0.0035$ \\
 \hline
 $\boldsymbol{2}$ & $0.0042$ & $0.0030$ & $0.0026$ \\%[1ex] 
 \hline
  $\boldsymbol{10}$ & $0.0039$ & $0.0030$ & $0.0027$ \\%[1ex] 
 \hline
\end{tabular}
\end{center}
\end{table}

\section{Conclusion}

The mathematical model of $n^{th}$-order chemical reactions of the type ``$A \to $ product'' carried out in a dispersed flow tubular reactor (DFTR) has been studied. The dynamics are described by a nonlinear parabolic partial differential equation with feedback boundary control applied through variations of the inlet concentration. The abstract problem in operator form is presented and studied. A feedback control design is proposed to ensure the exponential stability of the steady-state solution to the problem. The existence and uniqueness of the solution to the initial value problem are proved under the proposed control. 

The presented simulation results confirm the exponential decay of $w(x,t)$ for large values of $t$. Considering the chosen physical parameters, the profile of  $w(x,t)$ closely approximates the steady-state value at  $t=400\,s$, within an acceptable numerical tolerance.

The estimates of $\lambda_N$, presented in Table~\ref{table}, exhibit nonlinear behavior with respect to each of the parameters $n$ and $\alpha$ over the given range. It is important to note that the theoretical decay rate estimate $\lambda_T=\frac{v^2}{16D_{ax}}$, obtained from the differential inequality~\eqref{decay-rate}, provides a lower bound for the estimates $\lambda_N$ in~\eqref{lambda_N}.
Up to this point, there is no numerical evidence to indicate how accurate or sharp the estimate 
$\lambda_T$ in~\eqref{exp_st} is, as the data of Table~\ref{table} only addresses solutions with the specified initial condition~\eqref{initial}.
 We leave the analysis of the sharpness of  $\lambda_T$ for further studies. Another prospective topic for future research could focus on the input-to-state stability analysis of the problem expressed in equation~\eqref{dftr}.

\addcontentsline{toc}{section}{References}
%\bibliographystyle{gost2008}
%\bibliographystyle{plainurl}
%\bibliography{ref.bib}

\end{document}